\documentclass{amsart}

\usepackage{xcolor}
\usepackage{graphicx}            
\usepackage{epstopdf}     

\newtheorem{theorem}{Theorem}[section]

\theoremstyle{definition}
\newtheorem{definition}[theorem]{Definition}
\newtheorem{example}[theorem]{Example}

\newtheorem{prop}[theorem]{Proposition}
\newtheorem{lem}[theorem]{Lemma}
\newtheorem{cor}[theorem]{Corollary}

\theoremstyle{remark}
\newtheorem{remark}[theorem]{Remark}
\def\dd{{\rm d}}

\numberwithin{equation}{section}

\begin{document}

\title{Information Measures for Entropy and Symmetry }


\author{Daniel Lazarev}
\address{Department of Mathematics, Massachusetts Institute of Technology, Cambridge, MA 02139}
\curraddr{}
\email{dlazarev@mit.edu}
\thanks{}

\subjclass[2020]{Primary 28D20. Secondary 28C10, 94A17}

\date{}

\dedicatory{}

\commby{}

\begin{abstract}
Entropy and information can be considered dual: entropy is a measure of the subspace defined by the information  constraining the given ambient space. 
Negative entropies, arising in naïve extensions of the definition of entropy from discrete to continuous settings, are byproducts of the use of probabilities, which only work in the discrete case by a fortunate coincidence. 
We introduce the notions of sup-normalization and information measures, which allow for the appropriate generalization of the definition of entropy that keeps with the interpretation of entropy as a subspace volume. 
 Applying this in the context of topological groups and Haar measures, we elucidate the relationship between entropy, symmetry, and uniformity. 
\end{abstract}

\maketitle


\section{Introduction}

Entropy is popularly construed as a measure of disorder. 
Here we show that it can also be seen as a measure of uniformity and symmetry.
As shown explicitly below, when unconstrained, entropy maximization yields a uniform, equal-weight, or translation invariant distribution. 
Meanwhile, information lowers entropy, imposing certain nonuniformities. 
Symmetries, if viewed as uniformities across certain transformations or translations, would then naturally correspond to states of higher entropy. 
Thus, the ubiquity of symmetries in physics may not point to a delicate and precarious balance mysteriously maintained by Nature, but to a state of maximal entropy given the information imposed by the external forces or constraints.
When stricter constraints are imposed on a system, certain symmetries naturally break, but the system tends to the most ``uniform" or ``symmetric" distribution while satisfying those new constraints.
Entropy, symmetry, and information are ubiquitous precisely because they are fundamental.
Even in a purely mathematical sense, the symmetries present in a certain theory can be viewed as dual to the axiomatic constraints defining that theory.
Thus, while not the subject of this work, defining an entropy to compare the symmetries of different theories may be a way to classify those theories, similar to the Erlangen program in geometry \cite{Kle1893,KleHas1893,EilMac45}. 


Contradictory heuristics and claims have appeared in the scientific literature regarding the relationship between entropy and symmetry. For example, while \cite{Lin96a} asserts that entropy is (positively) correlated with measures of symmetry and similarity, others claim that symmetry lowers entropy \cite{Bor20}.

These views can be reconciled in light of the duality of entropy and information. 
If the symmetry is in the information—for example, by imposing that certain symmetry-obeying states be equal—then this certainly lowers entropy since it constrains the space to a particular subspace exhibiting that particular symmetry. However, having more symmetries in the state space, in the sense of giving equal weight to more points in the space, {\it uniformizes} the space, increasing its entropy. 


Put differently, one needs less information to encode a system that has many symmetries than one with little or no symmetries, and less information means more entropy, but  imposing symmetric information (restricting to particular symmetric states) lowers entropy, as would any information.

In the mathematical literature, the Haar measure has already been shown to yield the unique maximal entropy for groups using the  measure theoretic definition of entropy in dynamical systems theory \cite{AdlWei67,Ber69}. 
Moreover, for an essential spanning forest of an infinite graph $G$, it was shown in \cite{She06} that there exists a unique $\textrm{Aut}(G)$-invariant measure maximizing the specific entropy. 
Proper attention to the constraints and symmetries on the space of solutions to given PDEs also allows maximum entropy to be used as a selection criterion for underspecified PDEs, such as the Navier-Stokes equation \cite{GliLazChe20}.

\section{Entropy  as a Subspace Volume}
Shannon's information theoretic entropy is given by 
\begin{equation*}
S_{\rm disc}({\bf p}) = - \sum_{i=1}^n  p_i \log p_i  , 
\end{equation*}
where  ${\bf p} = \{ p_i \}_{i=1}^n$ is a discrete probability distribution \cite{Sha48, CovTho91}.
Shannon extended this to continuous random variables by a similar formula, which matches the thermodynamic entropy defined by Boltzmann and Gibbs. This is given by,
\begin{equation*}
S_{\rm cont} (p(x),{\mathcal{V}})
= - \int_{{\mathcal{V}}} p(x)  \log p(x)  \;\dd x  ,
\end{equation*}
where $p(x)$ is a probability density function on a set $\mathcal{V}$  \cite{Whe91,CovTho91,GavCheBec17}.
A generalization of these definitions is the Baron-Jauch, or relative entropy \cite{BarJau72,Ska75,Whe91}, 
which is defined measure theoretically and works for both continuous and discrete spaces.

\begin{definition}
 The \textit{relative entropy}
of a  probability  measure $\mu$  relative to a measure $\nu$
on a measure space $\mathcal{X}$ (where $\mu$ is 
absolutely continuous with respect to $\nu$, {\it i.e.}, $\mu \ll \nu$) is given by,
\begin{equation}
\label{eq:rel-def}
S_\nu(\mu,\mathcal{X})
= -\int_{\mathcal{X}}\left(\dfrac{\dd\mu}{\dd \nu} \right)  \log \left(\dfrac{\dd\mu}{\dd \nu} \right) \; \dd \nu .
\end{equation}
\end{definition}
This entropy depends on the reference measure $\nu$, and
 satisfies properties such as monotonicity, concavity, and subadditivity  \cite{Whe91}, as expected. 
The formula (\ref{eq:rel-def})
can be extended to define the entropy for a finite but not necessarily unit mass measure.
\begin{lem}
The relative entropy of a finite measure $\eta$ relative to a measure $\nu$ on a measure space
$\mathcal{X}$, with $\eta \ll \nu$, is given by
\begin{equation}
 \label{eq:fin-rel-def}
S_\nu(\eta,\mathcal{X}) 
= \log \eta(\mathcal{X}) 
 - \frac{1}{\eta({\mathcal{X}})} \int_{{\mathcal{X}}}\left(\dfrac{\dd \eta}{\dd \nu} \right)
 \log \left(\dfrac{\dd\eta}{\dd \nu} \right)  \dd \nu  .
\end{equation}

\begin{proof}
Immediate from setting $ \frac{\dd \mu}{\dd \nu}=\frac{1}{\eta(\mathcal{X})} \frac{\dd \eta}{\dd \nu}$ in (\ref{eq:rel-def}).
\end{proof}
\end{lem}

\begin{prop}
\label{prop:ent-max}
The uniform measure for the reference measure $\nu$ is given by $\frac{\dd \mu^*}{\dd \nu} = \frac{1}{\nu(\mathcal{X})}$, and its entropy by $S_\nu(\mu^*,{\mathcal{X}}) = \log \nu(\mathcal{X})$.
The entropy  is uniquely maximized by the uniform probability measure for $\nu$:
\begin{equation}
\label{eq:prop-1}
 S_\nu(\mu,{\mathcal{X}}) \leq  S_\nu(\mu^*,{\mathcal{X}}).
\end{equation}
More generally, for any finite measure, $\eta$, entropy is maximized by the reference measure itself:
\begin{equation}
\label{eq:prop-1_2}
 S_\nu(\eta,{\mathcal{X}}) \leq  S_\nu(\nu,{\mathcal{X}}), 
\end{equation}
where, again, $S_\nu(\nu,{\mathcal{X}}) = \log \nu(\mathcal{X})$.
\end{prop}
\begin{proof}
Setting $\frac{\dd \mu^*}{\dd \nu} = \frac{1}{\nu(\mathcal{X})}$ in (\ref{eq:rel-def}) gives $S_\nu(\mu^*,{\mathcal{X}}) = \log \nu(\mathcal{X})$ by a direct computation. Similarly, using (\ref{eq:fin-rel-def}) and setting $\eta = \nu$, we similarly obtain $S_\nu(\nu,{\mathcal{X}}) = \log \nu(\mathcal{X})$.
For the second part, starting from Eq. (\ref{eq:rel-def}), we denote $\mu'=\dd\mu/\dd \nu$
and use the concavity of the logarithm and Jensen's inequality, which gives,
\begin{align*}
\nonumber
S_\nu(\mu,{\mathcal{X}}) &=\int_{{\mathcal{X}}}\mu' \log \left(\dfrac{1}{\mu'} \right) \; \dd \nu \\
\nonumber
&\leq  \log \int_{{\mathcal{X}}} \mu' \left(\dfrac{1}{ \mu'} \right) \dd \nu\\
\nonumber
&= \log \nu({\mathcal{X}}) \\
&= S_\nu(\mu^*,{\mathcal{X}}) .
\end{align*}
The entropy (\ref{eq:rel-def}) is concave so the maximum we found is a global maximum. Thus among all measures, only the uniform attains the maximum possible value of entropy. 
Similarly, starting from (\ref{eq:fin-rel-def}) and  denoting $\eta' = \dd\eta/\dd \nu$, 
\begin{align*}
S_\nu(\eta,{\mathcal{X}}) &= \log \eta(\mathcal{X}) 
 + \frac{1}{\eta({\mathcal{X}})} \int_{{\mathcal{X}}} \eta' 
 \log \left(\dfrac{1}{\eta'} \right)  \dd \nu  \\
&\leq \log \eta(\mathcal{X}) 
 + \log \left[    \frac{1}{\eta({\mathcal{X}})} \int_{{\mathcal{X}}}    \eta' 
  \left(\dfrac{1}{ \eta'} \right)  \dd \nu   \right] \\
&= \log \eta(\mathcal{X}) 
 + \log   \frac{\nu( \mathcal{X}  )}{\eta({\mathcal{X}})} \\
&= \log \nu({\mathcal{X}}) \\
&=S_\nu(\nu,{\mathcal{X}}),
\end{align*}
so, as expected, the same maximum entropy is attained regardless of the definition used ((\ref{eq:rel-def}) or (\ref{eq:fin-rel-def})). 
\end{proof}
The definitions of entropy (\ref{eq:rel-def}) or (\ref{eq:fin-rel-def}) define the (logarithm of the) volume of a subspace of $\mathcal{X}$ based on the information or constraints defined by $\mu$ (in (\ref{eq:rel-def}))  or $\eta$ (in (\ref{eq:fin-rel-def})). Naturally, when no new information is introduced (uniform measure in (\ref{eq:rel-def}) or $\eta = \nu$ in (\ref{eq:fin-rel-def})), the space is not constrained, so we recover the (log) volume of the entire space, $ \log \nu(\mathcal{X})$, which is naturally the maximal value. 

In the natural extension of the definition of entropy to continuous systems, negative values of entropy also become possible. This is because, while discrete probabilities always fall within $[0,1]$, probability densities can also take  values in $(1, \infty)$. While an entropy of $0$ in the discrete case means we have maximal information, that is, enough to single out a single state (for example, think of a weighted die that always lands on $6$, so $S(\{p_1, \ldots, p_6\}) = S(\{  0,\ldots,0,1  \}) =\log 1 = 0$), an entropy $S <0$ in the continuous case indicates an arbitrary level of precision in our knowledge, seeing as we can in principle be as precise as an infinitesimal range of values for our state.

Nevertheless, this extension is unnatural in that it strays from entropy's primary and most powerful interpretation as a volume of available state space after the introduction of some information (constraints). We argue that this inconsistency arises because the true basis for the definition of entropy is not a probability density but a weight or membership function, which always takes values in $[0,1]$ but does not necessarily need to have unit mass, much like the functions defining fuzzy sets  
\cite{Zad65,Zad68}. In the discrete case, this distinction is trivial since discrete probabilities always take values in $[0,1]$, but it becomes significant when we generalize to continuous systems.

Let 
$\frac{\dd \mu}{\dd \nu}: \mathcal{X} \rightarrow [0,1]$
for some set $\mathcal{X}$. 
We can define a {\it weight function}, $\varphi := -\log \frac{\dd \mu}{\dd \nu} \in [0,\infty]$. 
Then $e^{-\varphi}=\frac{\dd \mu}{\dd \nu}  \in [0,1]$ and   $\mu(\mathcal{X}) = \int_\mathcal{X} \frac{\dd \mu}{\dd \nu} \dd \nu  
=  \int_\mathcal{X} e^{-\varphi} \dd \nu$, so using definition  (\ref{eq:fin-rel-def}) for the entropy of a finite measure $\mu$ relative to $\nu$, we have
\begin{equation}
\label{eq:weight-ent}
S_\nu (\mu, \mathcal{X}) = \log  \int_\mathcal{X} e^{-\varphi} \dd \nu + \frac{1}{ \int_\mathcal{X} e^{-\varphi} \dd \nu}  \int_\mathcal{X} \varphi \, e^{-\varphi} \, \dd \nu =: S (\varphi, \mathcal{X}) \, .
\end{equation}
For every point in $\mathcal{X}$, the {\it information function} $e^{-\varphi} \in [0,1]$ gives the degree to which that point is a member of the constrained subspace. 
Written this way, we see that $e^{-\varphi} \dd \nu$ is a weighted measure, and the entropy (\ref{eq:weight-ent}), which is nonnegative,\footnote{\;Of course this may fail to hold if $\mu(\mathcal{X}) < 1$; more on this below.} can be interpreted as the (log) volume of the constrained subspace plus the weighted average of the weight function $\varphi$  itself.\footnote{\;The second term in (\ref{eq:weight-ent}) can also be seen as measuring the ``softness" or ``leakiness" of a constraint, that is, the extent to which the constraint or information is enforced.  With ``hard" constraints, sets are clearly demarcated because the weight function is either $0$ or $1$—a given point is either in the set or is not. With softer constraints, points can leak out or bleed through: the constraint is more a suggestion than a rule. Softer constraints contribute to the entropy because, intuitively, they define a subset that is more uniform.}

Notice that for $\varphi = 0$ (no constraint or information, which happens when $\mu = \nu$), we have 
$S (0, \mathcal{X}) = \log \nu(\mathcal{X})$. 
Thus, every point is equally a member of the subspace ($e^\varphi = e^0 =1$), so all points are included, giving the (log) volume of the entire space. 
In fact, this holds for $\varphi = a$, for any constant  $a \geq 0$, since
\begin{align*}
S (a, \mathcal{X}) &= \log  \int_\mathcal{X} e^{-a} \dd \nu + \frac{1}{ \int_\mathcal{X} e^{-a} \dd \nu}  \int_\mathcal{X} a \, e^{-a} \, \dd \nu \\
&= \log \big(e^{-a} \nu (\mathcal{X})\big) + \frac{1}{e^{-a} \nu(\mathcal{X})} a e^{-a} \nu(\mathcal{X}) \\
&= -a +\log \nu (\mathcal{X}) + a \\
&= \log \nu (\mathcal{X}) .
\end{align*}
As we will later see, $\varphi = a$ corresponds to Haar measures, and the constant $e^{-a}$ is in fact nothing more than the arbitrary multiplicative constant up to which Haar measures are unique. Thus, at least heuristically, we see that the entropy for Haar measures is the same (and is the maximum, in fact), regardless of the arbitrary multiplicative factor.

On the other end, for $\varphi \rightarrow \infty$ (maximal constraint or information, which happens when $\frac{\dd \mu}{\dd \nu} =0$) on a constrained subset of unit mass ($ \int_\mathcal{X} e^{-\varphi} \dd \nu=1$), $S_\nu(\mu, \mathcal{X}) = 0$, as expected.


\section{Information Measures}


We assume that all groups are locally compact topological groups with continuous group action. 

Our main goal in this section is to show how to consistently compare the information encoded in different measures, and in so doing define the class of measures whose entropies will always correspond to the subspace volumes defined by the information in those respective measures. To that end, we first define the following.  

\begin{definition}
Two measures $\rho$ and $\xi$ on a measure space $(\mathcal{X}, \nu)$ (where $\nu$ is the {\it reference measure}) are {\it sup-normalized}  if $\rho, \xi \ll \nu$ and 
\begin{equation}
\label{eq:sup-norm}
\displaystyle \sup_{x \in \mathcal{X}} \frac{\dd \rho}{\dd \nu} (x) = \displaystyle \sup_{x \in \mathcal{X}} \frac{\dd \xi}{\dd \nu} (x)\, .
\end{equation}
\end{definition}
Of course, two finite measures of bounded density can always be sup-normalized by multiplying them by  appropriate constants. 

Intuitively, sup-normalization equates the peak weights each measure assigns points in the measure space relative to the reference measure. 
Doing so allows one to control for arbitrary scaling differences between different measures.
For example, since two Haar measures on a given group are unique up to a multiplicative constant,  the notion of sup-normalization serves as an appropriate condition for these Haar measures to be ``consistent".
Let $\mu_G$ and $\mu_G'$ be Haar measures for the topological group $G$.
If  $\mu_G$ and $\mu_G'$ are sup-normalized (regardless of which one serves as the reference measure), then $\mu_G(A) = \mu_G'(A)$ for all $A \subset G$, so the multiplicative constant relating them equals one. Sup-normalizing two Haar measures is also natural because, as we saw at the end of the last section, the entropies for two Haar measures are the same regardless of the multiplicative constant relating them. 


The next proposition establishes some basic results regarding sup-normalized measures. 
\begin{prop}
\label{prop:max-set}
Let $G$ be a  group, and $\rho$ and $\xi$ sup-normalized measures on $G$, with reference measure $\nu$. 
Then,\\
($i$) for all compact $A \subset G$, if $\nu = \mu_G$ (the sup-normalized Haar measure on $G$),
$$
\displaystyle
\dfrac{\displaystyle \sup_{g \in G} \rho(gA)}{\mu_G (A)}  \leq  \sup_{x \in G} \frac{\dd \rho}{\dd \mu_G} (x), \;\;\;\,\,  \dfrac{ \displaystyle\sup_{g \in G} \xi(gA)}{\mu_G (A)}  \leq  \sup_{x \in G} \frac{\dd \xi}{\dd \mu_G} (x)   \, ;
$$
\\
($ii$) if $\xi =  C \; \nu$, 
 then $\frac{\dd \rho}{\dd \xi} \leq 1$, and  $\rho (A) \leq \xi(A)$. 
\end{prop}
\begin{proof}
Writing $c:= \sup_{x \in G} \frac{\dd \rho}{\dd \mu_G} (x) =  \sup_{x \in G} \frac{\dd \xi}{\dd \mu_G} (x)$, we have, for every $A \subset G$, 
\begin{align}
\label{ineq1}
\sup_{g \in G} \rho(gA) &= \sup_{g \in G} \int_{gA} \frac{\dd \rho}{\dd \mu_G} (x) \; \dd \mu_G (x) \\
\nonumber
&\leq c  \sup_{g \in G} \int_{gA}  \dd \mu_G (x) \\
\nonumber
&=  c  \sup_{g \in G} \mu_G (gA)\\
\nonumber
&= c \; \mu_G(A) .
\end{align}
The result for $\xi$ is entirely analogous, which proves ($i$).

For ($ii$), we first note that by sup-normalization, $C = \sup_{x \in G} \frac{\dd \rho}{\dd \nu} = \sup_{x \in G} \frac{\dd \xi}{\dd \nu}$. Using $\xi =  C \nu$ and sup-normalization we have
\begin{align*}
\sup_{x \in G} \frac{\dd \rho}{\dd \nu} &= C \\
\sup_{x \in G}\left( \frac{\dd \rho}{\dd \xi}   \frac{\dd \xi}{\dd \nu} \right)&= C \\
C \sup_{x \in G}\left( \frac{\dd \rho}{\dd \xi}  \right)&= C \\
\sup_{x \in G} \frac{\dd \rho}{\dd \xi}  &= 1\; , 
\end{align*}
so $\frac{\dd \rho}{\dd \xi} \leq 1$. Finally, 
$$
\rho(A) = \int_A \frac{\dd \rho}{\dd\xi} \dd \xi \leq \int_A \dd \xi = \xi(A) ,
$$
completing the proof.
\end{proof}
\begin{remark} 
\label{rem1}
We  take $c=1$ for simplicity. Notice that, by statement ($i$) of Proposition \ref{prop:max-set}, for a reference measure that is Haar, the variation in $\dd \rho/\dd \mu_G$ is due entirely to that of $\rho$, and in this sense measures the new information introduced by $\rho$ on $(G, \mu_G)$ (see Figure \ref{fig1}). 
In other words, the translation invariance of $\mu_G$ ensures that the only variation in $\frac{\dd \rho}{\dd \nu}$  is due to that in $\rho$ (and similarly for $ \frac{\dd \xi}{\dd \nu}$).


With this choice of reference measure, statement ($ii$) in Proposition \ref{prop:max-set} becomes: 
If $\xi = \mu_G$ (where $\mu_G$ is a Haar measure on $G$), then $\rho(A) \leq \mu_G(A)$ and $\frac{\dd \rho}{\dd \mu_G} \leq 1$.
Thus, sup-normalized measures on $(G, \mu_G)$ encode information by defining a subset of the measure space  that gives different weights to different points in the space, as compared to a reference measure that gives the same  weight to all the points. 
\end{remark}

\begin{figure}
\includegraphics[width=8cm]{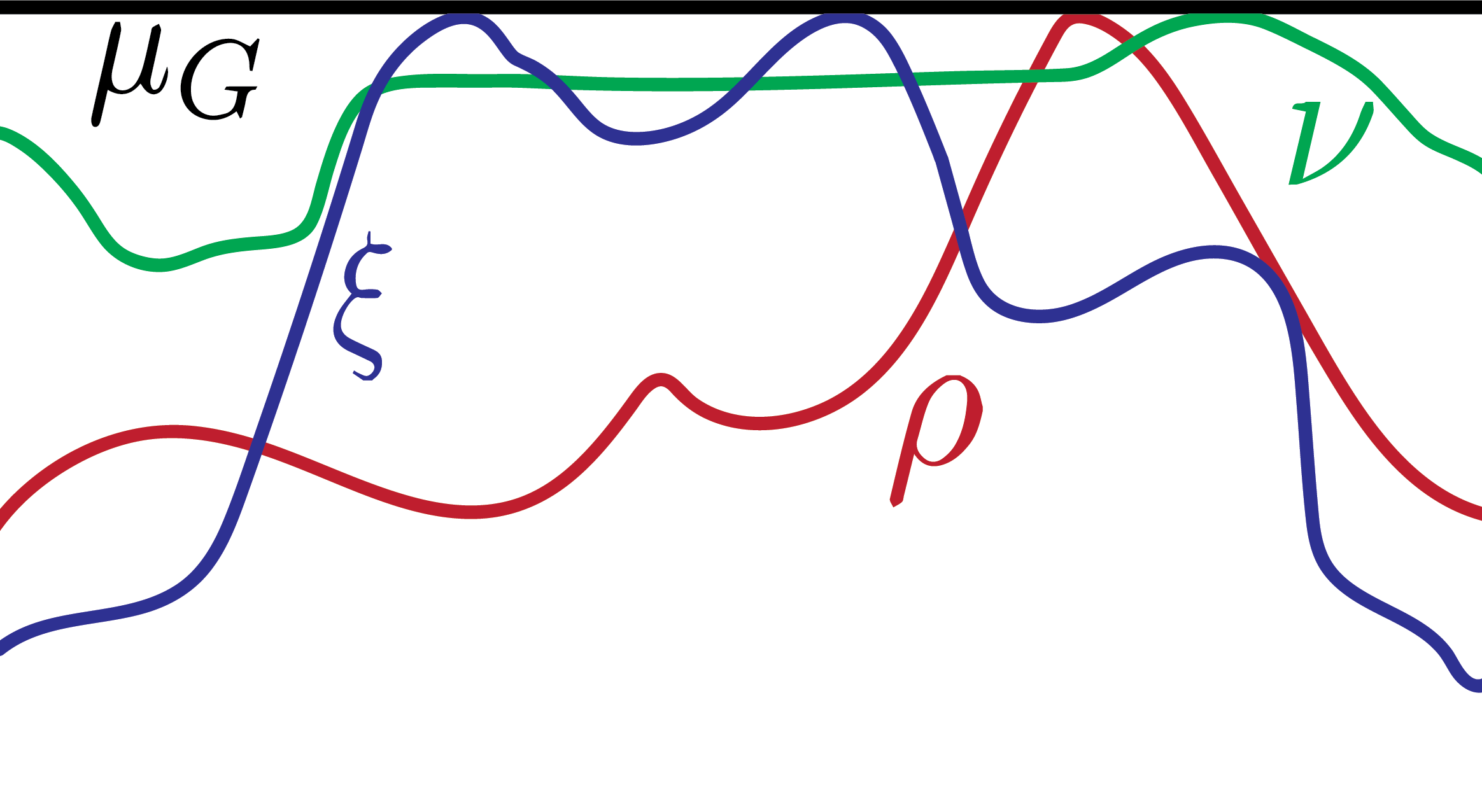}
\caption{Three measures, $\rho$, $\xi$, and $\nu$, sup-normalized with respect to the Haar measure $\mu_G$.}
\label{fig1}
\end{figure}

These results motivate us to define the following, which also allows us to explicitly distinguish the measures we need from probability measures.
\begin{definition}
An {\it information measure} is a measure $\rho$ on a measure space $(\mathcal{X}, \nu)$ (where $\nu$ is the {\it reference measure}) such that $\rho \ll \nu$  and $\dfrac{\dd \rho}{\dd \nu} \in [0,1]$. We call $\dfrac{\dd \rho}{\dd \nu} $ the {\it information function}. 
\end{definition}

Proposition \ref{prop:max-set} shows that information measures naturally arise as the sup-normalized measures on a given measure space. They are so named because they provide  information that allows the definition of some subspace of the given measure space. Moreover, as we will see explicitly below, measures assigning the same weight to every point on the measure space are the most uninformative, or entropy-maximizing. 
Since information measures need  not  necessarily be of unit mass, we will  use the entropy as defined for finite measures by formula (\ref{eq:fin-rel-def}).

We end this section with the following nonnegativity result for the entropy of any information measure.
\begin{lem}{(Nonnegativity.)}
\label{lem:nonneg}
For any information measure $\mu$ on a measure space $(\mathcal{X}, \nu)$, if $\mu(\mathcal{X}) \geq 1$, then,
$$
S_\nu (\mu , \mathcal{X}) \geq 0 .
$$
For $\mu(\mathcal{X}) \leq 1$, entropy is nonnegative if
$-\int_\mathcal{X} \frac{\dd \mu}{\dd \nu} \log  \left( \frac{\dd \mu}{\dd \nu} \right) \dd \nu \geq - \mu(\mathcal{X}) \log \mu(\mathcal{X})$.
\end{lem}
\begin{proof}
We write formula (\ref{eq:fin-rel-def}) as, 
$$
S_\nu (\mu , \mathcal{X}) =- \int_{\mathcal{X}} \left( \frac{1}{\mu(\mathcal{X}) } \frac{\dd \mu}{\dd \nu} \right) 
\log  \left( \frac{1}{\mu(\mathcal{X}) } \frac{\dd \mu}{\dd \nu} \right) \dd \nu .
$$
Clearly, the entropy is positive if $\frac{1}{\mu(\mathcal{X}) } \frac{\dd \mu}{\dd \nu} \leq 1$. Since $\mu$ is an information measure, $\frac{\dd \mu}{\dd \nu} \leq 1$, so $\mu(\mathcal{X}) \geq 1$ guarantees nonnegativity. 

For $\mu(\mathcal{X}) \leq 1$, the condition follows from formula \ref{eq:fin-rel-def} by requiring $S_\nu (\mu , \mathcal{X}) \geq 0$.  The condition is trivially satisfied for $\mu(\mathcal{X}) \geq 1$, since in that case $- \mu(\mathcal{X}) \log \mu(\mathcal{X}) \leq 0$ but we always have $-\int_\mathcal{X} \frac{\dd \mu}{\dd \nu} \log  \left( \frac{\dd \mu}{\dd \nu} \right) \dd \nu \geq 0$ for information measures.
\end{proof}

Given that by Proposition  \ref{prop:ent-max}    the entropy $S_\nu(\mu, \mathcal{X})$ is bounded above by $\log \nu(\mathcal{X})$,  we definitely need $\nu(\mathcal{X}) \geq 1$, since otherwise the entropy would be bounded above by a negative number. 
Since $\mu$ is an information measure,  $\mu(\mathcal{X}) \geq 1$ implies that  $\nu(\mathcal{X}) \geq 1$ as well, since by Proposition \ref{prop:max-set},  $\mu(\mathcal{X}) \leq \nu(\mathcal{X})$. Negative entropies only occur  for $\mu(\mathcal{X}) \leq 1$, and even then only if (rewriting the condition in terms of information functions as in the previous section) 
$\dfrac{1}{ \int_\mathcal{X} e^{-\varphi} \dd \nu}  \int_\mathcal{X} \varphi \, e^{-\varphi} \, \dd \nu \leq -\log  \int_\mathcal{X} e^{-\varphi} \dd \nu$, that is, if the entropy gained from the softness of the constraint is less than the (log) volume of the subspace. 

For our purposes, however, $\mu(\mathcal{X}) \geq 1$ is enough, including information measures that are also probability measures, or those that define larger subsets. 
I hope to address the $\mu(\mathcal{X}) \leq 1$ case in further detail in future work. 

\section{Entropy, Symmetry, and Haar measures}

This section is devoted to proving the following theorems:

\begin{theorem}
\label{thm:gen-ineq}
Let  $\mu_G$ be a Haar measure on a group $G$, and $\rho$ and $\xi$  information measures on $(G, \mu_G)$. Then for all compact  $A \subset G$,
\begin{equation}
\label{eq:gen-ineq}
S_\xi(\rho, A) \leq S_{\mu_G} (\rho, A) \leq S_{\mu_G} (\mu_G, A) ,
\end{equation}
where $S_{\mu_G} (\mu_G, A) = S_{\mu_G} (\mu_G, g  A) =\log \mu_G(A)$.
\end{theorem}

\begin{theorem}
\label{thm:rel-sym}
Let $H, G$ be any two topological  groups with $H \leq G$, 
and let $\mu_H, \mu_G$  be the corresponding Haar measures, respectively.
Let $\xi$ be any information measure such that $\xi \ll \mu_H$ and $\xi \ll \mu_G$.
If the Haar measures are sup-normalized, then for all $H \leq A \leq G$,
\begin{align}
\label{eq:thm-1a}
S_{\mu_H}(\xi, H)  &\leq S_{\mu_G}(\xi, A) , 
\end{align}
with equality if and only if $A = H$.
\end{theorem}

Theorem \ref{thm:gen-ineq} provides several important insights. 
First, together with Proposition \ref{prop:max-set} and Remark \ref{rem1}, it further justifies why, for any group, the most natural choice of reference measure is the Haar measure.  
Intuitively, the entropy with a reference measure that is not Haar is always bounded above by the entropy with a Haar reference measure because, for the former, the reference measure itself defines a subspace based on the information encoded by that particular reference measure. 
Second, Theorem \ref{thm:gen-ineq} shows that the entropy is maximized by the Haar measure, which takes the role of the uniform measure for the given group, making this maximum entropy invariant under translations by any group element. 
This makes rigorous several notions that were already useful in the physics community, but seem to have hitherto received only heuristic justifications (see, {\it e.g.}, \cite{DunTalHan07, GliLazChe20}). 


Theorem \ref{thm:rel-sym} allows the comparison of the entropies of different measures that are invariant under specific symmetries,  showing that measures that are ``more symmetric," in the sense that they are invariant under more symmetries, have higher entropies. 
Moreover, it shows that no matter the information measure $\xi$, the entropy for a group relative to its Haar measure will always be greater than that of any of its subgroups with their respective Haar measures. 

Before proving Theorem \ref{thm:rel-sym} we will establish a couple of partial results to motivate it.
For discrete groups, the Haar measure is simply the counting measure, so a particular instance of Theorem \ref{thm:rel-sym} follows almost immediately. 
For example, consider the dihedral groups $D_n, D_m$, where $n \leq m$. 
The entropies for these groups satisfy $S(D_n) = \log 2n \leq \log 2m = S(D_m)$. 
This is the most basic application of Theorem \ref{thm:rel-sym}.
In more general terms, we have the following. 
\begin{lem}
 Let $H, G$ be any two discrete topological groups with $H \leq G$, and let $\nu$ be the counting measure (cardinality). Then
\begin{equation}
\label{lemm-2}
S_\nu(\nu,H) = \log \nu(H) \leq \log \nu(G) = S_\nu(\nu,G) .
\end{equation}
\end{lem}
\begin{proof}
By Proposition \ref{prop:ent-max}, $S_\nu(\nu, X) = \log \nu(X)$, and the statement immediately follows since $H \leq G$. 
\end{proof}

If both groups are subgroups of some larger group, we also have the following.

\begin{prop}
\label{prop:1}
 Let $H, G \leq \mathcal{X}$ be two topological  groups,  $\mu_H, \mu_G$  their respective (normalized) Haar measures, and $\nu$  a Haar measure on the topological group $\mathcal{X}$. 
If $0<\nu(H) \leq \nu(G)$, then
\begin{equation}
\label{eq:thm-3}
S_\nu (\mu_H, H) \leq S_\nu (\mu_G, G) .
\end{equation}
\end{prop}
\begin{proof}
Since $G, H \subset \mathcal{X}$ and $\mu_G, \mu_H$ and $\nu$ are Haar measures, we have that $\dd \mu_G = \frac{1}{\nu(G)} \dd \nu, \; \dd \mu_H = \frac{1}{\nu(H)} \dd \nu$. 
Then 
\begin{align*}
S_{\nu} (\mu_H,H) &= - \int_{H}  \frac{\dd \mu_H}{\dd \nu}  \log \left(\frac{\dd \mu_H}{\dd \nu} \right) \; \dd \nu \\
&= - \int_{H} \frac{1}{\nu(H)}   \log \left(\frac{1}{\nu(H)}  \right) \; \dd \nu \\
&= \log \nu(H) , 
\end{align*}
and similarly, $S_{\nu} (\mu_G,G) = \log \nu(G)$.
Therefore, 
$$
S_{\nu} (\mu_H,H)= \log \nu(H)  \leq  \log \nu(G) =S_{\nu} (\mu_G,G),
$$
completing the proof.
\end{proof}

We now show that entropy is invariant under translation by a group element $g$ if both the measure and reference measure are invariant under it. 
\begin{prop}
\label{prop:invariant-ent}
Let $G$ be a locally compact topological group with Haar measure $\mu_G$. For any element $g \in G$ and any (not necessarily compact) subset $A \subset G$, let $\rho$ be an information measure. Then
\begin{equation}
S_{\mu_G} (\rho, A) \leq S_{\mu_G} (\mu_G, A) = S_{\mu_G} (\mu_G, g  A) .
\end{equation}
\end{prop}
\begin{proof}
By definition \ref{eq:fin-rel-def},
$S_{\mu_G} (\mu_G,A) 
= \log \mu_G (A)$.
Because  $\mu_G$ is a Haar measure, we have that $\mu_G (g  A) = \mu_G (A)$, so
\begin{align*}
S_{\mu_G} (\mu_G, g  A) 
= \log \mu_G (g  A) 
= \log \mu_G (A) 
= S_{\mu_G} (\mu_G, A) .
\end{align*}
As in the proof of Proposition \ref{prop:ent-max}, we have by Jensen's inequality, 
\begin{align*}
\nonumber
S_{\mu_G} (\rho,A) &=\int_{A  }  \frac{\dd \rho}{\dd \mu_G}  \log \left(\frac{\dd \mu_G}{\dd \rho} \right) \; \dd \mu_G \\
&\leq  \log \int_A     \frac{\dd \rho}{\dd \mu_G}  \frac{\dd \mu_G}{\dd \rho}   \dd \mu_G   \\
&= \log \mu_G (A) \\
&= S_{\mu_G} (\mu_G, A),
\end{align*}
which completes the proof.
\end{proof}

\begin{cor}
\label{cor:invariant-ent}
Given the setup in Proposition \ref{prop:invariant-ent}, we have that
\begin{equation}
S_{\mu_G} (\rho, g  A) \leq S_{\mu_G} (\mu_G, g  A) = S_{\mu_G} (\mu_G, A) .
\end{equation}
\end{cor}
\begin{proof}
As in the proof of Proposition \ref{prop:invariant-ent}, Jensen's inequality gives that $S_{\mu_G} (\rho, g  A) \leq S_{\mu_G} (\mu_G, g  A)$, and by Proposition \ref{prop:invariant-ent}, $S_{\mu_G} (\mu_G, g  A) = S_{\mu_G} (\mu_G, A)$.
\end{proof}

Therefore, while the entropy of a given information measure is not invariant under translations by a group element, it is always bounded above by the entropy of the Haar measure, which is invariant under any group element translations. 
The Haar measure of a topological group thus serves as the uniform measure for that group. 
Moreover, the entropy with respect to the measure that is ``more symmetric" is higher.

We include some examples to illustrate the results Proposition \ref{prop:invariant-ent} and Corollary \ref{cor:invariant-ent}.
\begin{example}
Let $G = (\mathbb{R}, +)$, so the Haar measure is the Lebesgue measure $\nu$, and let $A = [a,b] \subset \mathbb{R}$ (with $b>a$). Then $S_\nu (\nu, A) = \log(b-a)$, 
and for any $g \in \mathbb{R}$, we have $S_\nu (\nu, g+A) = \log(g+b-(g+a)) = \log(b-a) = S_\nu (\nu, A)$.
\end{example}

\begin{example}
Let $H = (\mathbb{R}^+, \times)$,  $B = [a,b] \subset \mathbb{R}^+$ (with $b>a$). 
Then the Haar measure on $H$ is given by $\mu_H(X) = \int_X \frac{1}{x} \dd x$. 
Using formula (\ref{eq:fin-rel-def}), we have 
\begin{align*}
S_{\mu_H} (\mu_H, B)  =  \log \mu_H(B) 
= \log \left(   \int_a^b  \frac{1}{x} \dd x  \right) 
=\log \log(b/a) .
\end{align*}
Then for any $g \in \mathbb{R}$ we have 
\begin{align*}
S_{\mu_H} (\mu_H, gB) &= \log \log \left( \frac{gb}{ga} \right) 
= \log \log(b/a)  
= S_{\mu_H} (\mu_H, B) ,
\end{align*}
as per Proposition \ref{prop:invariant-ent}.
\end{example}

\begin{example}
For comparison, we compute the entropy of $\mu_H$ relative to $\nu$ from the previous two examples.
\begin{align*}
S_\nu (\mu_H, B) &= \int_B \frac{\log x}{x} \dd x\\
&= \frac{1}{2}(\log^2 (b) - \log^2 (a)) \\
&= \frac{1}{2}(\log (b) - \log (a)) (\log (b) + \log (a)) \\
&= \frac{1}{2}\log (b/a) \log (ab)  ,
\end{align*}
which is clearly not invariant under addition or multiplication by an arbitrary group element $h \in \mathbb{R}^+$.

\end{example}


\begin{lem}
\label{lem:change-vars}
{\rm (Change of Reference Measure Formula.)}
For any measures $\mu, \nu$ on a set $X$ and information measure $\rho$ such that $\rho \ll \mu$ and $\rho \ll \nu$, we have, 
\begin{equation}
\label{eq:change-vars}
 S_{\nu} (\rho, X) = S_\mu (\rho, X)  - \frac{1}{\rho(X)} \int_X \log \left(\frac{\dd \mu}{\dd \nu} \right)  \dd \rho .
\end{equation}
\end{lem}
\begin{proof}
Starting from the chain rule for Radon-Nikodym derivatives, we have that
\begin{align*}
\frac{\dd \rho}{\dd \nu} &= \frac{\dd \rho}{\dd \mu}  \frac{\dd \mu}{\dd \nu} \\
\log \left(\frac{\dd \rho}{\dd \nu} \right)&= \log \left(\frac{\dd \rho}{\dd \mu} \right) + \log  \left(\frac{\dd \mu}{\dd \nu} \right) \\
\log \rho(X) -\frac{1}{\rho(X)} \int_X \log \left( \frac{\dd \rho}{\dd \nu} \right)  \dd \rho &= \log \rho(X)  - \frac{1}{\rho(X)} \int_X  \log \left( \frac{\dd \rho}{\dd \mu}\right) \dd \rho \\
 &\qquad- \frac{1}{\rho(X)} \int_X \log \left( \frac{\dd \mu}{\dd \nu}  \right) \dd \rho ,
\end{align*}
and the result follows by the definition of entropy (\ref{eq:fin-rel-def}). See also the proof of Lemma 7.2 in \cite{AmbGigSav14} for why we do not need to assume that $\mu \ll \nu$ for this to hold.
\end{proof}

\begin{proof}[Proof of Theorem \ref{thm:gen-ineq}.]
By Lemma \ref{lem:change-vars} we have that 
\begin{equation*}
S_{\mu_G} (\rho, A) = S_\xi (\rho, A)  - \frac{1}{\rho(X)} \int_A \log \left(\frac{\dd \xi}{\dd \mu_G} \right)  \dd \rho .
\end{equation*}
Then by Proposition \ref{prop:max-set} and Remark \ref{rem1} we have that $\frac{\dd \xi}{\dd \mu_G} \leq 1$, so 
\begin{equation}
\label{eq:gap}
0 \leq - \frac{1}{\rho(X)} \int_A \log \left( \frac{\dd \xi}{\dd \mu_G} \right) \dd \rho = S_{\mu_G} (\rho, A) - S_\xi (\rho, A),\end{equation}
which proves the first inequality.
Incidentally, formula (\ref{eq:gap}) also gives the exact ``entropic gap" between the reference measures 
$\xi$ and $\mu_G$.
The rest of the results are given by Proposition \ref{prop:invariant-ent} and Corollary  \ref{cor:invariant-ent}. 
\end{proof}

\begin{proof}[Proof of Theorem \ref{thm:rel-sym}.]
Recalling that $H \leq G$, 
we first note that for all $A \subset H$, we have $\mu_H = \mu_G$ by sup-normalization. 
Therefore, $\frac{\dd \xi}{\dd \mu_H} =  \frac{\dd \xi}{\dd \mu_G}$ for all $A \subset H$, so $S_{\mu_G} (\xi, H) =S_{\mu_H} (\xi, H)$.
The result then follows by noting that $S_{\mu_G} (\xi, X)\geq 0$ for all $X \subset G \setminus H$ (or by invoking the monotonicity of entropy, {\it i.e.}, that $S_{\mu_G} (\xi, H)  \leq S_{\mu_G} (\xi, G)$ given that $H \leq G$).  
\end{proof}

\section*{Acknowledgments}
The author would like to thank  Henry Cohn for his mentorship and feedback, and David Ebin for his valuable comments on the draft. 

 \bibliography{refs} 
 \bibliographystyle{alpha}

\end{document}